\documentclass[reqno,12pt]{amsart}
\usepackage{xr}
\usepackage{graphicx}
\usepackage[usenames,dvipsnames,svgnames,table]{xcolor}
\usepackage{amsrefs}
\usepackage{epstopdf}
\usepackage{booktabs}
\usepackage{amssymb,amsmath}
\usepackage{amsthm}
\usepackage{enumerate}
\newtheorem{theorem}{Theorem}[section]
\newtheorem{proposition}[theorem]{Proposition}
\newtheorem{lemma}[theorem]{Lemma}
\newtheorem{definition}[theorem]{Definition}
\newtheorem{corollary}[theorem]{Corollary}

\newtheorem{example}[theorem]{Example}
\numberwithin{equation}{section}

\newcommand{\R}{{\mathbb R}}

\newcommand{\CA}{{\mathcal{A}}}

\newcommand{\CC}{{\mathcal{C}}}

\newcommand{\CL}{{\mathcal{L}}}

\newcommand{\CO}{{\mathcal{O}}}
\newcommand{\CS}{{\mathcal{S}}}

\renewcommand{\ll}{{\langle}}
\newcommand{\rr}{{\rangle}}

\newcommand{\p}{\mathfrak{p}}
\renewcommand{\t}{\mathfrak{t}}

\newcommand{\Z}{{\mathbb Z}}

\begin{document}
\title[Asymptotics]{Asymptotic distributions associated to piecewise quasi-polynomials}

\author{Paul-Emile Paradan}
\address{Institut Montpelli\'erain Alexander Grothendieck, CNRS ,
Universit\'e de Montpellier} \email{paul-emile.paradan@umontpellier.fr }

\author{Mich{\`e}le Vergne}
\address{ Universit\'e Denis Diderot, Institut Math\'ematique de
Jussieu, Sophie Germain}
\email{michele.vergne@imj-prg.fr}

\maketitle
%

%
%
%
%
%
%
%

\section{Introduction}

Let $V$ be a finite dimensional real vector space equipped with a lattice
$\Lambda$.
Let $P\subset V$ be a rational polyhedron.
The Euler-Maclaurin formula (\cite{guisteriemann}, \cite{ber-ver-elm}) gives an asymptotic estimate, when $k$ goes to $\infty$,
for the Riemann sum $\sum_{\lambda\in k P\cap \Lambda } \varphi(\lambda/k)$
of the values of a test function $\varphi$ at the sample points
$\frac{1}{k}\Lambda\cap P$ of $P$, with leading term $k^{\dim P}\int_P \varphi$.
Here we consider the slightly more general case of a weighted  sum.
Let $q(\lambda,k)$ be a quasi-polynomial function on $\Lambda\oplus \Z$.
We consider, for $k\geq 1$,  the distribution
$$\ll\Theta(P;q)(k),\varphi\rr=\sum_{\lambda\in kP\cap \Lambda} q(\lambda,k) \varphi(\lambda/k)$$
 and we show (Proposition \ref{pro:tech2}) that the function $k\mapsto \ll \Theta(P;q)(k),\varphi\rr$ admits an asymptotic expansion  when $k$ tends to
$\infty$ in powers of $1/k$ with coefficients periodic functions of $k$.

We extend this result to an algebra  $\CS(\Lambda)$
of piecewise quasi-polynomial functions
 on $\Lambda\oplus \Z\subset V\oplus \R$.
  A function $m(\lambda,k)$ ($\lambda\in \Lambda, k\in \Z$) in $\CS(\Lambda)$
  is supported in  an union of polyhedral cones in $V\oplus \R$.
The main feature of a function $m(\lambda,k)$ in $\CS(\Lambda)$ is that $m(\lambda,k)$ is
entirely determined by its large behavior in $k$.
We associate to $m(\lambda,k)$ a formal series $A(m)$ of distributions on $V$
encoding the asymptotic behavior of $m(\lambda,k)$ when $k$ tends to $\infty$.

The motivating  example
 is the case where $M$ is a projective  manifold, and
 $\CL$ the corresponding ample bundle.
If $T$ is a torus acting on $M$, then   write, for $t\in T$,
$$\sum_{i=0}^{\dim M} (-1)^i {\rm Tr}(t,H^i(M,\CO(\CL^k)))=\sum_{\lambda}m(\lambda,k)t^\lambda$$
where $\lambda$ runs over the lattice $\Lambda$  of characters of $T$.
The corresponding asymptotic expansion of the distribution
$\sum_{\lambda}m(\lambda,k)\delta_{\lambda/k}$
is an important object associated to $M$
involving the Duistermaat-Heckmann measure and the Todd class of $M$,
 see \cite{ver:graded} for its determination.
The  determination of similar asymptotics in the more general case of twisted Dirac operators is the object of
a forthcoming article \cite{pep-vergneasy}.

Thus let $m\in \CS(\Lambda)$, and consider the sequence
  $$\Theta(m)(k)=\sum_{\lambda\in \Lambda}m(\lambda,k)\delta_{\lambda/k}$$
 of distributions on $V$ and  its asymptotic expansion $A(m)$ when $k$ tends to $\infty$.
Let $T$ be the torus with lattice of characters $\Lambda$.
If $g\in T$ is an element of finite order, then $m^g(\lambda,k):=g^{\lambda}m(\lambda,k)$ is again in $\CS(\Lambda)$.
 Our main result (Theorem \ref{theo:unique}) is that the piecewise quasi-polynomial function $m$ is entirely determined by the collections of asymptotic expansions $A(m^g)$, when $g$ varies over the set of elements of $T$ of finite order.

We also  prove (Proposition \ref{pro:wanted})
 a functorial property of $A(m)$ under pushforward.

 We use these results to give new proofs of functoriality of the formal quantization of a symplectic manifold \cite{parformal1} or, more generally, of a spinc manifold \cite{parformal2}.

 For these applications, we also consider the case where $V$ is a Cartan subalgebra
 of a compact Lie group, and anti-invariant distributions on $V$ of a similar nature.

\subsection{Piecewise polynomial functions}\label{subsecPW}

Let $V$  be a real vector space equipped with a lattice $\Lambda$.
Usually, an element of $V$ is denoted by $\xi$, and an element of $\Lambda$ by $\lambda$.
In this article, a cone $C$ will always be a closed convex polyhedral cone, and $0\in C$.

Let $\Lambda^*$ be the dual lattice, and let $g\in T:=V^*/\Lambda^*$.
If $G\in V^*$ is a representative of $g$ and $\lambda\in \Lambda$,
then we denote $g^{\lambda}=e^{2i\pi \ll G,\lambda\rr}$.

A periodic function $m$ on $\Lambda$ is a function such that there exists a positive integer $D$ (we do not fix $D$) such that $m(\lambda_0+D \lambda)=m(\lambda_0)$ for $\lambda,\lambda_0\in \Lambda$.
The space of such functions is linearly generated by the functions
$\lambda\mapsto g^{\lambda}$ for $g\in T$ of finite order.
By definition, the algebra of quasi-polynomial functions on
$\Lambda$ is generated by polynomials and  periodic functions on $\Lambda$.
If $V_0$ is a rational subspace of $V$, the restriction of $m$ to $\Lambda_0:=\Lambda\cap V_0$ is a quasi-polynomial function on $\Lambda_0$.
The space of quasi-polynomial functions is graded:
 a quasi-polynomial  function   homogeneous of  degree $d$
 is a linear combination of functions
$t^{\lambda}  h(\lambda)$
where $t\in T$ is of finite order, and $h$ an homogeneous polynomial on $V$ of degree $d$.
Let $q(\lambda)$ be a quasi-polynomial function on $\Lambda$.
There is a sublattice $\Gamma$ of $\Lambda$ of finite index $d_\Gamma$ such that
for any given $\gamma\in \Lambda$, we have
$q(\lambda)=p_\gamma(\lambda)$ for any $\lambda\in \gamma+\Gamma$
where $p_\gamma(\xi)$  is a (uniquely determined) polynomial function on
$V$.
Then define $q_{pol}(\xi)=\frac{1}{d_\Gamma}\sum_{\gamma\in \Lambda/\Gamma}p_\gamma(\xi)$,
a polynomial function on $V$.
This polynomial function is independent of the choice of the sublattice $\Gamma$.
Then $q(\lambda)-q_{pol}(\lambda)$ is
 a linear combination of functions of the form
$t^{\lambda}h(\lambda)$ with  $h(\lambda)$ polynomial and $t\neq 1$.

Using the Lebesgue measure associated to $\Lambda$, we identify generalized functions on $V$ and distributions  on $V$.
If $\theta$ is a generalized
function on $V$, we may write  $\int_{V}\theta(\xi)\varphi(\xi)d\xi$
for its value on the test function $\varphi$.
If $R$ is a rational affine subspace of $V$, $R$ inherits a canonical translation invariant measure. If $P$ is a rational polyhedron in $V$, it generates a rational affine subspace of $V$, and $\int_P \varphi$ is well defined for $\varphi$ a smooth function with compact support.

We say that a distribution $\theta(k)$ depending of an integer $k$ is periodic in $k$ if there exists a positive integer $D$ such that for any test function $\varphi$ on $V$, and $k_0,k\in \Z$,
$\ll\theta(k_0+D k),\varphi \rr=\ll\theta(k_0),\varphi\rr$.
Then there exists (unique) distributions $\theta_{\zeta}$
indexed by $D$-th roots of unity such that  $\ll\theta(k),\varphi\rr=
\sum_{\zeta,\zeta^D=1}\zeta^k \ll\theta_{\zeta},\varphi\rr.$

Let $(\Theta(k))_{k\geq 1}$  be a sequence of
 distributions.
 We  say that $\Theta(k)$ admits an asymptotic expansion (with periodic coefficients)
 if there exists  $n_0\in \Z$ and a sequence of  distributions $\theta_n(k),{n\geq 0}$, depending periodically of $k$, such that for any test function $\varphi$ and any non negative  integer $N$,  we have
$$
\ll \Theta(k),\varphi\rr=
k^{n_0}\sum_{n=0}^N \frac{1}{k^n}\ll\theta_n(k), \varphi\rr + o(k^{n_0-N}).
$$
 We write
$$
\Theta(k)\equiv k^{n_0}\sum_{n=0}^{\infty} \frac{1}{k^n}\theta_n(k).
$$

The distributions $\theta_n(k)$ are uniquely determined.

Given a sequence $\theta_n(k)$ of periodic distributions, and $n_0\in \Z$,
we write formally $M(\xi,k)$  for the series  of distributions on $V$ defined by
$$\ll M(\xi,k),\varphi\rr=k^{n_0}
\sum_{n=0}^{\infty} \frac{1}{k^{n}}\int_V\theta_{n}(k)(\xi)\varphi(\xi)d\xi.$$
We can multiply $M(\xi,k)$ by quasi-polynomial functions $q(k)$ of $k$
and smooth functions $h(\xi)$ of $\xi$
and obtain the formal series $q(k)h(\xi)M(\xi,k)$ of the same form with $n_0$ changed to $n_0+{\rm degree}(q)$.

Let $E=V\oplus \R$, and we consider the lattice
 $\tilde \Lambda=\Lambda\oplus \Z$ in $E$.
 An element of $\tilde \Lambda$ is written as
 $(\lambda,k)$ with $\lambda\in \Lambda$ and $k\in \Z$.
We consider quasi-polynomial functions
$q(\lambda,k)$ on $\tilde \Lambda$.
As before, this space is graded. We call the degree of a quasi-polynomial function
on $\Lambda\oplus \Z$ the total degree.
A quasi-polynomial function $q(\lambda,k)$ is of total degree $d$ if it is a linear combination of functions
$(\lambda,k)\mapsto j(k)  t^{\lambda} k^a h(\lambda)$
where $j(k)$ is a periodic function of $k$, $t\in T$  of finite order, $a$ a non negative integer, and $h$ an homogeneous polynomial on $V$ of degree $b$, with $b$ such that  $a+b=d$.

Let $q(\lambda,k)$ be a quasi-polynomial function on $ \Lambda\oplus \Z$.
We construct $q_{pol}(\xi,k)$ on $V\times \Z$, and depending polynomially on $\xi$ as before. We choose a sublattice of finite index $d_\Gamma$ in $\Lambda$ and
functions $p_\gamma(\xi,k)$  depending polynomially on $\xi\in V$ and quasi-polynomial in $k$
 such that $q(\lambda,k)=p_\gamma(\lambda,k)$
 if $\lambda\in \gamma+\Gamma$.
 Then $q_{pol}(\xi,k)=\frac{1}{d_\Gamma}\sum_{\gamma\in \Lambda/\Gamma}p_\gamma(\xi,k)$.
We say that $q_{pol}(\xi,k)$ is the polynomial part (relative to $\Lambda$) of $q$.
If $q$ is homogeneous of total degree $d$, then the function
$(k,\xi)\mapsto q_{pol}(k\xi,k)$ is a linear combination of functions of the form
$j(k) k^d s(\xi)$ where $j(k)$ is a periodic function of $k$ and $s(\xi)$ a polynomial function of $\xi$.

\begin{proposition}\label{pro:tech}
Let $P$ be a rational polyhedron in $V$ with non empty interior.
Let $q(\lambda,k)$ be a quasi-polynomial function on $\Lambda\oplus \Z$ homogeneous of total degree $d$. Let $q_{pol}(\xi,k)$ be its polynomial part.
Let $k\geq 1$.
The distribution
$$\ll \Theta(P;q)(k),\varphi\rr=\sum_{\lambda\in kP}  q(\lambda,k) \varphi(\lambda/k)$$
admits an asymptotic expansion when $k\to \infty$  of the form
$$k^{\dim V} k^d\sum_{n=0}^\infty \frac{1}{k^{n}}\ll\theta_{n}(k),\varphi\rr.$$
Furthermore,  the term $k^{d}\ll\theta_{0}(k),\varphi\rr$ is given by
$$k^d\ll\theta_{0}(k),\varphi\rr=\int_P q_{pol}(k\xi,k) \varphi(\xi) d\xi$$
where $q_{pol}$ is the polynomial  part (with respect to $\Lambda$)   of $q$.
\end{proposition}
\begin{proof}
Let $q(\lambda,k)=j(k) k^a g^{\lambda}h(\lambda)$ be a quasi-polynomial function of total degree $d$.
Let
\begin{equation}\label{eq:Theta0}
\ll \Theta_0^g(P)(k),\varphi\rr=\sum_{\lambda\in kP\cap \Lambda} g^{\lambda}  \varphi(\lambda/k).
\end{equation}
If $\Theta_0^g(P)(k)$ admits the asymptotic expansion
$M(\xi,k)$,
 then $\Theta(P;q)(k)$ admits
 the asymptotic expansion
 $ j(k)k^a h(k\xi) M(\xi,k)$.
So it is sufficient to consider the case where $q(\lambda,k)=g^{\lambda}$
and the distribution $\Theta_0^g(P)(k)$.

We now proceed as in \cite{ber-ver-elm} for the case $g=1$ and sketch the proof.
By decomposing the characteristic function  $[P]$ of the polyhedron $P$ in a signed sum of characteristic functions of tangent cones, via the Brianchon Gram formula,
then decomposing furthermore each tangent cone in  a signed sum of cones $C_a$  of the form $\Sigma_a\times R_a$ with $\Sigma_a$ is a translate of a unimodular cone and $R_a$ a rational space,
we  are reduced to
study this distribution for the product of the dimension $1$  following situations.

$V=\R, \Lambda=\Z$ and one of the following two cases:

\begin{itemize}
\item
$P=\R$

\item
$P=s+\R_{\geq 0}$
with $s$ a rational number.

\end{itemize}

For example, if $P=[a,b]$ is an interval in $\R$ with rational end points $a,b$, we write
$[P]=[a,\infty]+[-\infty,b]-[\R]$.

For $P=\R$, and $\zeta$ a root of unity, it is easy to see that
$$\ll \Theta^\zeta(k),\varphi\rr=
\sum_{\mu\in \Z}  \zeta^\mu \varphi(\mu/k)$$
is equivalent to
$k\int_\R \varphi(\xi)d\xi$ if $\zeta=1$
or is equivalent to $0$ if $\zeta\neq 1$.

We now study the case where $P=s+\R_{\geq 0}$.
Let $$\ll \Theta^\zeta(k),\varphi\rr=
\sum_{\mu\in \Z, \mu-ks \geq 0}  \zeta^\mu \varphi(\mu/k)$$ and let us compute its asymptotic expansion.

For $r\in \R$,
the fractional part $\{r\}$ is defined by $\{r\}\in [0,1[, r-\{r\}\in \Z$.
If $\mu$ is an integer greater or equal to  $ks$, then
$\mu=ks+\{-ks\}+u$ with $u$ a non negative integer.

We consider first the case where $\zeta=1$.
This case has been treated for example in \cite{cohen} (Theorem 9.2.2), and there is an
Euler-Maclaurin formula with remainder which leads to the following asymptotic expansion.

The function  $z\mapsto \frac{e^{xz}}{e^z-1}$ has a simple pole at $z=0$. Its Laurent series at $z=0$ is

$$
\frac{e^{xz}}{e^z-1} =\sum_{n=-1}^\infty B_{n+1}(x)\frac{z^{n}}{(n+1)!}
 $$
where $B_n(x)$ $(n\geq 0)$ are the Bernoulli polynomials.

 If $s$ is rational, and $n\geq 0$,  the function $k\mapsto B_n(\{-ks\})$ is a periodic function of $k$  with period the denominator of $s$,
and
$$
 \sum_{\mu\in \Z, \mu\geq ks}\varphi(\frac{\mu}{k})\\\equiv
k( \int_s^\infty  \varphi(\xi) d\xi -
 \sum_{n=1}^{\infty} \frac{1}{k^n} \frac{B_n(\{-k s\})}{n!}\varphi^{(n-1)}(s)).$$
 This formula is easily proven by Fourier transform.
 Indeed,
 for $f(\xi)=e^{i\xi z}$,
 the series $\sum_{\mu\geq ks} f(\mu/k)$
 is
 $\sum_{u\geq 0}e^{is z} e^{i \{-k s\}z/k} e^{iuz/k}.$
 It is convergent if $z$ is in the upper half plane, and the sum is
 $$F(z)(k)=-e^{isz}\frac{e^{i\{-k s\} z/k}}{e^{iz/k}-1}.$$
 So the Fourier transform of the tempered distribution $\Theta^{\zeta=1}(k)$ is the boundary value
 of the holomorphic function $z\mapsto F(z)(k)$ above. We can compute the asymptotic behavior
 of $F(z)(k)$ easily when $k$ tends to $\infty$, since $\{-ks\}\leq 1$,
 and $z/k$ becomes small.

Rewriting $[P]$ as the signed sum of the characteristic functions of the cones $C_a$, we see  that the distribution $\Theta_0^g(P)(k)$ for $g=1$
is equivalent to $$k^{\dim V}(\sum_{n=0}^{\infty} \frac{1}{k^{n}} \theta_n(k))$$
with $\theta_0$ independent of $k$, and given by  $\ll\theta_0,\varphi\rr=\int_P \varphi(\xi)d\xi$.

Now consider the case where $\zeta\neq 1$. Then
$$\sum_{\mu\in \Z, \mu\geq ks}\zeta^{\mu}\varphi(\mu/k)=
 \sum_{u\geq 0}\zeta^{ks+\{-ks\}}\zeta^u \varphi(s+\{-ks\}/k+u/k).$$

The function $k\mapsto \zeta^{ks+\{-ks\}}$ is a periodic function of $k$ with period $ed$ if $\zeta^e=1$  and $ds$ is an integer.
If $\zeta\neq 1$, the function $z\mapsto \frac{e^{xz}}{\zeta e^z-1}$ is holomorphic at $z=0$. Define the polynomials $B_{n,\zeta}(x)$ via the Taylor series expansion:
$$
\frac{e^{xz}}{\zeta e^z-1} =\sum_{n=0}^\infty B_{n+1,\zeta}(x)\frac{z^{n}}{(n+1)!}.
 $$
It is easily seen by Fourier transform that
$\sum_{\mu\in \Z, \mu\geq ks}\zeta^{\mu}\varphi(\mu/k)$
is equivalent to
$$
-k\zeta^{ks+\{-ks\}}
 \sum_{n=1}^{\infty} \frac{1}{k^n} \frac{B_{n,\zeta}(\{-k s\})}{n!}\varphi^{(n-1)}(s).$$
In particular, $\Theta^\zeta(k)$ admits an asymptotic expansion  in non negative powers of $1/k$ and each coefficient of this asymptotic expansion is a periodic distribution supported at $s$.

Rewriting $[P]$ in terms of the signed cones $C_a$, we see that indeed if $g\in T$ is not $1$, one of the corresponding $\zeta$ in the reduction to a product of one dimensional  cones    is not $1$, and so
$$\Theta_0^g(P)(k)\equiv k^{\dim V-1}(\sum_{n=0}^{\infty} \frac{1}{k^{n}}\theta_n(k)).$$
So we obtain our proposition.

\end{proof}

Consider now $P$ a rational polyhedron, with possibly empty interior.
Let $C_P$ be the cone of base $P$ in $E=V\oplus \R$,
 $$C_P:=\{(t\xi,t), t\geq 0, \xi\in P\}.$$

Let $q(\lambda,k)$ be a quasi-polynomial function on $\Lambda\oplus \Z$.
We consider again
$$\ll \Theta(P;q)(k),\varphi\rr=\sum_{\lambda\in kP\cap \Lambda} q(\lambda,k) \varphi(\lambda/k).$$

Consider the vector space $E_P$ generated by the cone $C_P$ in $E$.
It is clear that $\Theta(P;q)$ depends only of the restriction $r$ of $q$ to
$E_P\cap (\Lambda\oplus \Z)$.
This is a quasi-polynomial function on $E_P$ with respect to the lattice
 $E_P\cap (\Lambda\oplus \Z)$. We assume that
 the quasi-polynomial function $r$ is homogeneous of degree $d_0$.
 This degree might be smaller that the total degree of $q$.
Consider the affine space $R_P$ generated by $P$ in $V$.
Let $E_P^{\Z}=E_P\cap (V\oplus \Z)$. If $\xi\in R_P, k\in \Z$, then
$(k\xi,k)\in E_P^{\Z}$.
We will see shortly (Definition \ref{def:poly}) that we can define
a function $(\xi,k)\mapsto r_{pol}(\xi,k)$  for $(\xi,k)\in E_P^{\Z}$,
and that the function $(\xi,k)\mapsto r_{pol}(k\xi,k)$ on $R_P\times \Z$
 is a linear combination of functions of the form $k^{d_0} j(k) s(\xi)$ where $j(k)$ is a periodic function of $k$
and $s(\xi)$ a polynomial function of $\xi$, for $\xi$ varying on the affine space $R_P$.

We now can state the general formula.

\begin{proposition}\label{pro:tech2}
Let $P$ be a rational polyhedron in $V$.
Let $q(\lambda,k)$ be a quasi-polynomial function on $\Lambda\oplus \Z$.
Let $r$ be its restriction to
$E_P\cap (\Lambda\oplus \Z)$ and $r_{pol}$ the "polynomial part" of $r$ on
$E_P\cap (V\oplus \Z)$.
Assume that the quasi-polynomial function $r$ is homogeneous of degree $d_0$.
Let $k\geq 1$. The distribution
$$\ll \Theta(P;q)(k),\varphi\rr=\sum_{\lambda\in kP}  q(\lambda,k) \varphi(\lambda/k)$$
admits an asymptotic expansion when $k\to \infty$ of the form
$$k^{\dim P} k^{d_0}\sum_{n=0}^\infty \frac{1}{k^{n}}\ll\theta_{n}(k),\varphi\rr.$$
Furthermore,  the term $k^{d_0}\ll\theta_{0}(k),\varphi\rr$ is given by
$$k^{d_0}\ll\theta_{0}(k),\varphi\rr=\int_P r_{pol}(k\xi,k) \varphi(\xi) d\xi.$$
\end{proposition}

\begin{proof}
We will reduce the proof of this proposition to the case treated before of a
 polyhedron with interior. Let ${\rm lin}(P)$ be the linear space parallel to $R_P$, and $\Lambda_0:=\Lambda \cap {\rm lin}(P)$.
If $R_P$ contains a point $\beta\in \Lambda$, then $E_P$ is isomorphic to
${\rm lin}(P)\oplus \R$ with lattice $\Lambda_0\oplus \Z$.
Otherwise, we will have to dilate $R_P$.
More precisely,
let $I_P=\{k\in \Z, kR_P\cap \Lambda\neq \emptyset\}$.
This is an ideal in $\Z$.
Indeed if $k_1\in I_P,k_2\in I_P$, $\alpha_1,\alpha_2\in R_P$ are such that $k_1\alpha_1\in \Lambda$, $k_2\alpha_2\in \Lambda$, then  $\alpha_{1,2}=\frac{1}{n_1k_1+n_2k_2}(n_1 k_1 \alpha_1+n_2 k_2 \alpha_2)$
is in $R_P$, and $(n_1k_1+n_2k_2)(\alpha_{1,2})\in \Lambda$.
Thus there exists a smallest $k_0>0$ generating the ideal $I_P$.
We see that our distribution $\Theta(P;p)(k)$
is identically equal to $0$ if $k$ is not in $I_P$.
Let $\delta_{I_P}(k)$ be the function of $k$ with

$$\delta_{I_P}(k)= \begin{cases}
0\hspace{25mm} {\rm if}\,  k\notin I_P ,\\
1 \hspace{25mm} {\rm if}\,  k=u k_0\in I_P.\\
    \end{cases}$$
 This is a periodic function of $k$ of period $k_0$.
We choose $\alpha\in R_P$ such that $k_0\alpha\in \Lambda$.
We identify $E_P$ to ${\rm lin}(P)\oplus \R$ by the map
$T_\alpha(\xi_0,t)= (\xi_0+ t k_0\alpha, t k_0).$
In this identification, the lattice $(\Lambda\oplus \Z)\cap E_P$
becomes the lattice $\Lambda_0\oplus \Z$.
Consider $P_0=k_0(P-\alpha)$, a polyhedron with interior in ${\rm lin}(P)$.
Let $q^{\alpha}(\gamma,u)=r(\gamma+uk_0\alpha,uk_0)$.
This is a quasi-polynomial function  on $\Lambda_0\oplus \Z$.
Its total degree is $d_0$.
We have defined its polynomial part $q_{pol}^{\alpha}(\xi,u)$ for $\xi\in {\rm lin}(P)$, $u\in \Z$.

\begin{definition}\label{def:poly}

Let $(\xi,k)\in E_P^{\Z}$.
Define:
$$r_{pol}(\xi,k)= \begin{cases}
0\hspace{45mm} {\rm if}\,  k\notin I_P ,\\
q_{pol}^{\alpha}(\xi-uk_0\alpha,u)\hspace{18mm} {\rm if}\,  k=u k_0\in I_P.\\
    \end{cases}$$

\end{definition}

The function
$r_{pol}(\xi,k)$ does not depend of the choice of $\alpha$.
Indeed, if $\alpha,\beta\in R_P$ are such that $k_0\alpha,k_0\beta\in \Lambda$,
then
$q^{\beta}(\gamma,u)=q^{\alpha}(\gamma+uk_0(\beta-\alpha),u).$
Then  we see that
$q_{pol}^{\beta}(\xi,u)=q_{pol}^{\alpha}(\xi+uk_0(\beta-\alpha),u)$.
Furthermore, the function
$(k,\xi)\mapsto r_{pol}(k\xi,k)$ is of the desired form, a linear combination of functions
$\delta_{I_P}(k)j(k) k^{d_0} s(\xi)$ with $s(\xi)$ polynomial functions on $R_P$.

If $\varphi$ is a test function on $V$, we define the test function $\varphi_0$ on ${\rm lin}(P)$
by $\varphi_0(\xi_0)=\varphi(\frac{\xi_0}{k_0}+\alpha)$.
We see that

\begin{equation}\label{eq:Thetazeroa}
\ll\Theta(P;q)(u k_0),\varphi\rr=\ll\Theta(P_0;q^{\alpha})(u),\varphi_0\rr.
\end{equation}
Thus we can apply Proposition \ref{pro:tech}.
We obtain
$$\ll\Theta(P;q)(u k_0),\varphi\rr\equiv u^{\dim P} u^{d_0}\sum_{n=0}^{\infty}
\frac{1}{u^n}\ll\omega_n(u),\varphi_0\rr.$$
We have
$$u^{d_0}\ll\omega_0(u),\varphi_0\rr=\int_{P_0}q_{pol}^{\alpha}(u\xi_0,u)\varphi_0(\xi_0)d\xi_0
=\int_{P_0}q_{pol}^{\alpha}(u\xi_0,u)\varphi(\frac{\xi_0}{k_0}+\alpha)$$
When $\xi_0$ runs in $P_0=k_0(P-\alpha)$, $\xi=\frac{\xi_0}{k_0}+\alpha$
runs over $P$. Changing variables, we obtain
$$u^{d_0}\ll\omega_0(u),\varphi_0\rr=k^{d_0}\int_P r_{pol}(k\xi,k)\varphi(\xi)d\xi.$$
Thus we obtain our proposition.
\end{proof}

\bigskip

Let $P$ be a rational polyhedron in $V$ and $q$ a quasi-polynomial function on $\Lambda\oplus \Z$.
We do not assume that $P$ has interior in $V$.
We denote by $[C_P]$ the characteristic function of  $C_P$.
Then the function $q(\lambda,k)[C_P](\lambda,k)$ is zero if $(\lambda,k)$ is not in $C_P$
or equal to $q(\lambda,k)$ if $(\lambda,k)$ is in $C_P$.
We denote it by $q[C_P]$.
The space of functions on $ \Lambda\oplus \Z$  we will study
is the following space.

\begin{definition}
We define the space $\CS(\Lambda)$
to be the space of functions on $ \Lambda\oplus \Z$
linearly generated by the functions $q[C_P]$  where $P$
 runs over rational polyhedrons in $V$ and $q$ over quasi-polynomial functions on $ \Lambda\oplus \Z$.
\end{definition}

The representation of $m$ as a sum of functions  $q[C_P]$ is not unique.
For example, consider $V=\R$, $P=\R, P_+:=\R_{\geq 0}, P_-:=\R_{\leq 0}, P_0:=\{0\}$,
then $[C_P]=[C_{P_+}]+[C_{P_-}]-[C_{P_0}]$.

\begin{example}\label{exa:imp}
An important example of functions $m\in \CS(\Lambda)$ is the following.
Assume that we have a closed cone $C$ in $V\oplus \R$, and a covering
$C=\cup_\alpha C_\alpha$ by closed  cones.
Let $m$ be a function on $C\cap (\Lambda\oplus \Z)$, and assume that the restriction of
$m$ to $C_\alpha\cap (\Lambda\oplus \Z)$
is given by a quasipolynomial function $q_\alpha$.
Then, using exclusion-inclusion formulae, we see that $m\in \CS(\Lambda)$.
\end{example}

\begin{definition}

If $m(\lambda,k)$ belongs to $\CS(\Lambda)$, and $g\in T$ is an element of finite order, then define
$$m^{g}(\lambda,k)=g^{\lambda}m(\lambda,k).$$
\end{definition}
The function $m^g$ belongs to
$\CS(\Lambda)$.

If $m\in \CS(\Lambda)$, and $k\geq 1$, we denote by
$\Theta(m)(k)$ the distribution on $V$ defined by
$$\ll\Theta(m)(k),\varphi\rr=
\sum_{\lambda\in \Lambda}m(\lambda,k) \varphi(\lambda/k),$$
if $\varphi$ is a test function on $V$.
The following proposition follows immediately  from Proposition \ref{pro:tech2}.

\begin{proposition}
If $m(\lambda,k)\in \CS(\Lambda)$, the distribution $\Theta(m)(k)$ admits an asymptotic expansion $A(m)(\xi,k)$.
\end{proposition}

The function $m(\lambda,k)$ can be non zero, while $A(m)(\xi,k)$ is zero.
For example let $V=\R$,  $P=\R$ and $m(\lambda,k)=(-1)^{\lambda}$.
Then $\Theta(m)(k)$ is the distribution on $\R$ given by ${\rm T}(k)=\sum_{\lambda=-\infty}^{\infty}(-1)^\lambda\delta_{\lambda/k}, k\geq 1$
and this
is equivalent to $0$.
However, here is an unicity theorem.

\begin{theorem}\label{theo:unique}
Assume that $m\in\CS(\Lambda)$ is such that
$A(m^g)=0$ for all $g\in T$ of finite order, then $m=0$.
\end{theorem}

\begin{proof}
We start by the case of a function $m=q [C_P]$ associated to a single polyhedron
$P$ and a quasi-polynomial function $q$.
Assume first that $P$ is with non empty interior $P^0$.
If $q$ is not identically $0$, we write
 $q(\lambda,k)=\sum_{g\in T} g^{-\lambda}p_g(\lambda,k)$ where $p_g(\lambda,k)$ are
 polynomials in $\lambda$.
 If $d$ is the total degree of $q$, then all the polynomials  $p_g(\lambda,k)$
 are  of
 degree less or equal than $d$.
We choose $t\in T$ such that $p_t(\lambda,k)$ is of degree $d$.
If we consider the quasi-polynomial $q^t(\lambda,k)$,
 then its polynomial part is $p_t(\lambda,k)$ and the homogeneous component
 $p_t^{top}(\lambda,k)$ of degree $d$ is not zero.
 We write $p_t^{top}(\xi,k)=
 \sum_{\zeta,a} \zeta^k k^a p_{\zeta,a}(\xi)$ where $p_{\zeta,a}(\xi)$
 is a polynomial in $\xi$ homogeneous of degree $d-a$.
Testing against a test function $\varphi$
 and computing the  term in $k^{d+\dim V}$ of the asymptotic expansion by Proposition \ref{pro:tech}, we see that
$\sum_{\zeta,a} \zeta^k k^d\int_P p_{\zeta,a}(\xi) \varphi(\xi)d\xi=0$.
 This is true for any test function $\varphi$.
 So, for any $\zeta$, we  obtain $\sum_a p_{\zeta,a}(\xi)=0$.
 Each of the $p_{\zeta,a}$ being homogeneous of degree $d-a$,
 we see that $p_{\zeta,a}=0$  for any $a,\zeta$.
 Thus $p_t^{top}=0$, a contradiction.
 So we obtain that $q=0$, and $m=q[C_P]=0$.
 Remark that to obtain this conclusion, we may use only test functions $\varphi$
 with support contained in the interior $P^0$ of $P$.

Consider now a general polyhedron $P$ and the vector space ${\rm lin}(P)$.
Let us prove that $m(\lambda,k)=q(\lambda,k)[C_P](\lambda,k)$ is identically $0$ if
$A(m^g)=0$ for any $g\in T$ of finite order.
Using the notations of the proof of Proposition \ref{pro:tech2},
we see that $m(\lambda,k)=0$, if $k$ is not of the form $ uk_0$.
Furthermore, if $q^\alpha(\gamma,u)=q(\gamma+uk_0\alpha,uk_0)$,
it is sufficient to prove that
$q^{\alpha}=0$.
Let $P_0=k_0(P-\alpha)$, a polyhedron with interior in $V_0$.
Consider $m_0=q^{\alpha}[C_{P_0}]$.
We consider $T$ as the character group of $\Lambda$, so $T$ surjects on
$T_0$.
Let $g\in T$ of finite order and such that $g^{k_0\alpha}=1$, and let
$g_0$ be the restriction of $g$ to $\Lambda_0$.
Using Equation \ref{eq:Theta0},
we then see that $$\ll\Theta(m^g)(u k_0),\varphi\rr=
\ll\Theta(m_0^{g_0})(u),\varphi_0\rr.$$

Any $g_0\in T_0$ of finite order is the restriction to
$\Lambda_0$ of an element $g\in T$ of finite order and such that
$g^{k_0\alpha}=1$.
So we conclude that the asymptotic expansion, when $u$ tends to $\infty$,  of
$\Theta(m_0^{g_0})(u)$ is equal to $0$ for any $g_0\in T_0$ of finite order.
Remark again that we need only to know that
$\ll\Theta(m^g)(k),\varphi\rr\equiv 0$
for test functions $\varphi$ such that the support $S$ of $\varphi$ is
contained in a very small neighborhood of compact subsets of
$P$ contained in the relative interior of $P^0$.

For any integer $\ell$,
 denote by $\CS_\ell(\Lambda)$ the subspace of functions $m\in \CS(\Lambda)$ generated by the functions $q[C_P]$ with $\dim P\leq \ell$.

When $\ell=0$, our polyhedrons are  a finite number of rational points $f\in V$, the function $m(\lambda,k)$ is supported on the union of lines $(ud_ff,ud_f)$ if $d_f$ is the smallest integer such that $d_ff$ is in $\Lambda$.
Choose a test function $\varphi$ with support near $f$.
Then $u\mapsto\ll\Theta(m)(d_f u),\varphi\rr$ is identical to its asymptotic expansion
  $m(ud_ff, ud_f) \varphi(f).$
  Clearly we obtain that $m=0$.

  If $m\in \CS_\ell(\Lambda)$ by inclusion-exclusion, we can write
  $$m=\sum_{P; \dim(P)=\ell}q_P[C_P]+\sum_{H, \dim H<\ell} q_H [C_H] $$
  and we can assume that the intersections of a polyhedron $P$
  occurring in the first sum, with any polyhedron $P'$  occurring in the decomposition of $m$ and different from $P$ is of dimension strictly less than $\ell$.
  Consider $P$ in the first sum, so  $\dim(P)=\ell$.
  We can thus choose test functions $\varphi$ with support in small neighborhoods of
  $K$, with $K$ a compact subset contained in the relative interior of $P$.
  Then $$\ll\Theta(m^g)(k),\varphi\rr=\ll\Theta(q_P^g[C_P])(k),\varphi\rr.$$
 The preceding argument shows that $q_P[C_{P}]=0$.
   So $m\in \CS_{\ell-1}(\Lambda)$.
   By induction $m=0$.
\end{proof}

\section{Composition of piecewise quasi-polynomial functions}
Let $V_0,V_1$  be  vector spaces with lattice $\Lambda_0,\Lambda_1$.

Let $C_{0,1}$ be a closed polyhedral rational cone in $V_0\oplus V_1$
(containing the origin).
Thus for any $\mu\in \Lambda_1$, the set of $\lambda\in V_0$ such that
$(\lambda,\mu)\in C_{0,1}$ is a rational polyhedron $P(\mu)$ in $V_0$.
Let $P$ be a polyhedron in $V$.
We assume that for any $\mu\in \Lambda_1$,
$P\cap P(\mu)$ is compact.
Thus,  for $m=q_P[C_P]\in \CS(\Lambda)$,
and $c(\lambda,\mu)$ a quasi-polynomial function  on
$\Lambda_0\oplus \Lambda_1$, we can compute

$$m_c(\mu,k)=\sum_{(\lambda,\mu)\in C_{0,1}} m(\lambda,k) c(\lambda,\mu).$$

\begin{proposition}\label{pro:wanted}
The function $m_c$ belongs to
$\CS(\Lambda_1)$.
\end{proposition}

Before establishing this result, let us give an example, which occur
for example in the problem of computing the multiplicity of a representation $\chi^\lambda\otimes \chi^\lambda$ of $SU(2)$ restricted to the maximal torus.

\begin{example}
Let $V_0=V_1=\R$, and
$\Lambda_0=\Lambda_1=\Z$. Let $P:=[0,2]$, and let
$$q(\lambda,k)=\begin{cases}
\frac{1}{2}(1-(-1)^\lambda) \hspace{1cm} {\rm if}\,\, 0\leq \lambda \leq 2 k\\
0 \hspace{4cm} {\rm otherwise}.\\
\end{cases}
$$

Let $$C_{0,1}=\{(x,y)\in \R^2; x\geq 0, -x\leq y\leq x\}$$
and $$c(\lambda,\mu)=\frac{1}{2}(1-(-1)^{\lambda-\mu}).$$

Let $\mu\geq 0$. Then
$$m_c(\mu,k)=\frac{1}{4}\sum_{0\leq \lambda\leq 2k, \lambda\geq \mu}
(1-(-1)^{\lambda})(1-(-1)^{\lambda-\mu})=\left( 1+ \left( -1 \right) ^{\mu} \right)  \left( k/2-\mu/4 \right).$$

So if $P_1=[0,2], P_2:=[-2,0], P_3:=\{0\}$,
we obtain

$$m_c=q_1[C_{P_1}]+q_2[C_{P_2}]+q_3 [C_{P_3}]$$

with

$$\begin{cases}q_1(\mu,k)=\left( 1+ \left( -1 \right) ^{\mu} \right)  \left( k/2-\mu/4 \right),\\
 q_2(\mu,k)=\left( 1+ \left( -1 \right) ^{\mu} \right)  \left( k/2+\mu/4 \right),\\
q_3(\mu,k)=-k.\\
\end{cases}
$$

\end{example}

We now start the proof of Proposition \ref{pro:wanted}.

\begin{proof}
Write $c(\lambda,\mu)$ as a sum of products of
quasi-polynomial functions $q_j(\lambda)$, $f_j(\mu)$, and $q_P(\lambda,k)$
a sum of products of quasi-polynomial functions $m_\ell(k)$, $h_\ell(\lambda)$.
Then we see that it is thus sufficient to prove that, for $q(\lambda)$ a quasi-polynomial function of $\lambda$, the function
\begin{equation}\label{eq sum}
S(q)(\mu,k)= \sum_{\lambda \in kP\cap P(\mu)} q(\lambda)
\end{equation}
belongs to $\CS(\Lambda_1)$.
For this, let us recall some results on families of polytopes $\p({\bf b})\subset E$ defined by linear inequations. See for example
\cite{ber-ver-passemuraille}, or \cite{sze-ver}.

Let $E$ be a vector space, and $\omega_i, i=1,\ldots,N$ be a sequence of linear forms on $E$. Let ${\bf b}=(b_1,b_2,\ldots, b_N)$ be an element of $\R^N$.
Consider the polyhedron $\p({\bf b})$ defined by the
inequations $$\p({\bf b})=\{v\in E; \ll\omega_i,v\rr\leq b_i, i=1,\ldots,N\}.$$
We assume $E$  equipped with a lattice $L$, and inequations
$\omega_i$  defined by elements of $L^*$.
Then if the parameters $b_i$  are in $\Z^N$, the polytopes $\p({\bf b})$ are rational convex polytopes.

Assume that there exists ${\bf b}$ such that $\p({\bf b})$ is compact (non empty).
Then $\p({\bf b})$ is compact (or empty) for any ${\bf b}\in \R^N$.
Furthermore, there exists a closed  cone $\CC$ in $\R^N$ such that
$\p({\bf b})$ is non empty if and only if ${\bf b}\in {\mathcal C}$.
There is a decomposition $\CC=\cup_\alpha \CC_\alpha$ of
$\mathcal C$ in  closed polyhedral cones with non empty interiors,
where  the polytopes $\p({\bf b})$, for ${\bf b}\in \CC_\alpha$, does not change of shape.
More precisely:

$\bullet$
 When $\bf b$  varies in the interior of $\CC_\alpha$,
the polytope $\p({\bf b})$
remains with the same number of vertices $\{s_1({\bf b}), s_2({\bf b}),\ldots, s_L({\bf b})\}$.

$\bullet$
for each $1\leq i\leq L$,
there exists a cone $C_i$ in $E$, such that
the tangent cone to the polytope $\p({\bf b})$ at the vertex $s_i({\bf b})$ is the affine cone
$s_i({\bf b})+C_i$.

$\bullet$
the map ${\bf b}\to s_i({\bf b})$ depends  of the parameter ${\bf b}$, via linear maps
$\R^N \to E$  with rational coefficients.

Furthermore -as proven for example in \cite{ber-ver-passemuraille}- the Brianchon-Gram decomposition
of $\p({\bf b})$ is "continuous" in ${\bf b}$ when $b$ varies on $C_\alpha$, in a sense discussed in \cite{ber-ver-passemuraille}.

Before continuing,
let us give a very simple example, let $b_1,b_2, b_3$ be $3$ real parameters
and consider $\p(b_1,b_2,b_3)=\{x\in \R, x\leq b_1, -x\leq b_2, -x\leq b_3\}.$
So we are studying the intersection of the interval $[-b_2,b_1]$ with the half line $[-b_3,\infty]$.
Then for $\p({{\bf b}})$ to be non empty, we need that ${\bf b}\in \CC$,
with
$$\CC=\{{{\bf b}};  b_1+b_2\geq 0, b_1+b_3\geq 0\}.$$

Consider $\CC=\CC_1\cup \CC_2$, with
$$\CC_1=\{{{\bf b}}\in \CC; b_2-b_3\geq 0\},$$
$$\CC_2=\{{{\bf b}}\in \CC; b_3-b_2\geq 0\}.$$

On $\CC_1$ the vertices of $\p({{\bf b}})$ are $[-b_3,b_1]$, while on
$\CC_2$ the vertices of $\p({{\bf b}})$ are $[-b_2,b_1]$.

The Brianchon-Gram decomposition of $\p({{\bf b}})$  for ${\bf b}$ in the interior of $\CC_1$ is $[-b_3,\infty]+[-\infty,b_1]-\R$.
If ${\bf b}\in \CC_1$ tends to the point $(b_1,b_2,-b_1)$ in the boundary of $\CC$, we see
the Brianchon-Gram decomposition tends to
that  $[b_1,\infty]+[-\infty,b_1]-\R$, which is indeed the polytope $\{b_1\}$.

Let  $q(\gamma)$ be a quasi-polynomial function of $\gamma\in L$.
Then, when ${\bf b}$ varies in $\CC_\alpha\cap \Z^N$, the function
$$S(q)({\bf b})=\sum_{\gamma\in \p({\bf b})\cap L} q(\gamma)$$
is given by a quasi-polynomial function of $b$.
This is proven in \cite{sze-ver}, Theorem 3.8. In this theorem, we sum an exponential polynomial function $q(\gamma)$ on the lattice points of $\p({\bf b})$ and obtain an exponential polynomial function of  the parameter $b$. However, the explicit formula shows that if we sum up a quasi-polynomial function of $\gamma$,  then we obtain a quasi-polynomial function of ${\bf b}\in \Z^N$.
Another proof follows from \cite{ber-ver-passemuraille} (Theorem 54) and the continuity of Brianchon-Gram decomposition.
In \cite{ber-ver-passemuraille},  only the summation of polynomial functions is studied, via a Brianchon-Gram decomposition, but the same proof gives the result for quasi-polynomial functions
(it depends only of  the fact that the vertices vary via rational linear functions  of ${\bf b}$).
The relations between partition polytopes $P_\Phi(\xi)$ (setting used in \cite{sze-ver}, \cite{ber-ver-passemuraille})
and families of polytopes $\p({\bf b})$ is standard, and
is explained for example in the introduction of \cite{ber-ver-passemuraille}.

Consider now our situation with $E=V$ equipped with the lattice $\Lambda$.
The polytope $kP\subset V$ is given by a sequence of inequalities
$\omega_i(\xi)\leq k a_i$, $i=1,\ldots, I$,
 where we can assume
  $\omega_i\in \Lambda^*$ and $a_i\in \Z$ by eventually multiplying by a large integer the inequality.
The polytope $P(\mu)$ is given by a sequence of inequalities
$\omega_j(\xi)\leq \nu_j(\mu)$, $j=1,\ldots,J$ where $\nu_j$ depends linearly on
$\mu$. Similarly we can assume $\nu_j(\mu)\in \Z$.
Let $$(\mu,k)\mapsto {{\bf b}}(\mu,k)=[ka_1,\ldots, ka_I, \nu_1(\mu),\ldots,\nu_J(\mu)]$$
a linear map from $\Lambda_1\oplus \Z$  to $\Z^N$.
Our polytope $kP\cap P(\mu)$ is the polytope $\p({{\bf b}}(k,\mu))$
and $$S(q)(\mu,k)=\sum_{\lambda\in \p({{\bf b}}(k,\mu))\cap \Lambda} q(\lambda)=S(q)({{\bf b}}(\mu,k)).$$

Consider one of the cones $\CC_\alpha$. Then ${{\bf b}}(\mu,k)\in \CC_\alpha$, if and only if  $(\mu,k)$ belongs to a rational polyhedral cone $C_\alpha$ in $V_1\oplus \R$.
If $Q$ is a quasi-polynomial function of ${\bf b}$, then
$Q({{\bf b}}(\mu,k))$ is a quasi-polynomial function of $(\mu,k)$.
 Thus on each of the cones $C_\alpha$, $S(q)(\mu,k)$ is given by a quasi-polynomial function of $(\mu,k)$.
 From Example \ref{exa:imp}, we conclude that $S(q)$ belongs to $\CS(\Lambda_1)$.
 \end{proof}

\section{Piecewise quasi-polynomial functions  on the Weyl chamber}
For applications, we have also to consider the following situation.

Let $G$ be a compact Lie group.
Let $T$ be a maximal torus of $G$, $\t$ its Lie algebra, $W$ be the Weyl group.
Let $\Lambda\subset \t^*$ be the weight lattice of $T$.
We choose a system $\Delta^+\subset \t^*$ of positive roots,
and let $\rho\in \t^*$ be the corresponding element.
 We consider the positive Weyl chamber $\t^*_{\geq 0}$ with interior
 $\t^*_{> 0}$.

We consider now $\CS_{\geq 0}(\Lambda)$ the space of functions generated by the functions  $q[C_P]$ with polyhedrons $P$ contained in $\t^*_{\geq 0}$.
This is a subspace of $\CS(\Lambda)$.
If $t\in T$ is an element of finite order, the function $m^t(\lambda,k)=t^{\lambda} m(\lambda,k)$
is again in
$\CS_{\geq 0}(\Lambda)$.

If $m\in \CS_{\geq 0}(\Lambda)$, we define the following anti invariant distribution
with value on a test function $\varphi$ given by
$$\ll\Theta_a(m)(k),\varphi\rr
=\frac{1}{|W|}\sum_{\lambda}
m(\lambda,k) \sum_{w\in W}\epsilon(w) \varphi(w(\lambda+\rho)/k)$$

\begin{proposition}
If for every $t\in T$ of finite order, we have
$\Theta_a(m^t)\equiv 0$, then $m=0$.
\end{proposition}

\begin{proof}
Consider $\varphi$ a test function supported in the interior of the Weyl chamber.
Thus, for $\lambda\geq 0$,  $\varphi(w(\lambda+\rho)/k)$
is not zero only if $w=1$.
So
$$\ll\Theta_a(m)(k),\varphi\rr=\frac{1}{|W|}\sum_{\lambda\geq 0} m(\lambda,k)  \varphi((\lambda+\rho)/k)$$
while $$\ll\Theta(k),\varphi\rr=\sum_{\lambda\geq 0} m(\lambda,k)  \varphi(\lambda/k).$$

Let $(\partial_\rho \varphi)(\xi)=\frac{d}{d\epsilon}`\varphi(\xi+\epsilon \rho)|_{\epsilon=0}$ and consider the series of differential operators
with constant coefficients $e^{\partial_\rho/k}=1+\frac{1}{k}\partial_\rho+\cdots$.
We then see that,
if $\ll A(\xi,k),\varphi \rr$ is the asymptotic expansion
of $\ll\Theta(k),\varphi\rr$,
the asymptotic expansion
of $\ll\Theta_a(k),\varphi\rr$
is $\ll  A(\xi,k), e^{\partial\rho/k}\varphi\rr.$
Proceeding as in the proof of Theorem \ref{theo:unique},
  we see that if $\ll\Theta_a(m^t)(k),\varphi\rr\equiv 0$ for all $t\in T$
   of finite order, then $m(\lambda,k)$ is identically $0$ when $\lambda$ is  on the interior of the Weyl chamber.

Consider all  faces (closed) $\sigma$ of the closed Weyl chamber.
Define $\CS_{\ell ,\geq 0}\subset \CS(\Lambda)$ to be the space of
$m=\sum_{\sigma, \dim(\sigma)\leq \ell} m_\sigma$,
where $m_\sigma\in \CS_{\geq 0}(\Lambda)$
is such that
$m_\sigma(\lambda,k)=0$ if $\lambda$ is not in $\sigma$.
Let us prove by induction on $\ell$ that if $m\in \CS_{\ell,\geq 0}$ and $\Theta_a^t(m^t)\equiv 0$, for all $t\in T$ of finite order, then
$m=0$.

If $\ell=0$, then $m(\lambda,k)=0$ except if $\lambda=0$, and our distribution is
$$m(0,k)\sum_{w} \epsilon(w)\varphi(w\rho/k).$$
Now, take for example $\varphi(\xi)=\prod_{\alpha>0}(\xi,H_\alpha) \chi(\xi)$ where
$\chi$ is invariant with small compact support and identically equal to $1$ near $0$.
Then $\ll\Theta_a(m),\varphi\rr$
for $k$ large  is  equal to $c\frac{1}{k^N}m(0,k)$,
where $N$ is the number of positive roots, and $c$ a non zero constant.
So we conclude that $m(0,k)=0$.

Now consider
$m=\sum_{\dim \sigma=\ell} m_\sigma+\sum_{\dim f<\ell}m_f$.
Choose $m_\sigma$ in the first sum.
Let $\sigma^0$ be the relative interior of $\sigma$.
Let $\Delta_0$  be the set of roots $\alpha$, such that $\ll H_\alpha,\sigma\rr=0$.
Then $\t^*=\t_1^*\oplus \t_0^*$, where
$\t_0^*=\sum_{\alpha\in \Delta^0}\R \alpha$
and $\t_1^*=\R\sigma$.
We write $\xi=\xi_0+\xi_1$ for $\xi\in \t^*$, with $\xi_0\in \t_0^*,\xi_1\in \t_1^*$.
Then $\rho=\rho_0+ \rho_1$
with $\rho_1\in\t_1^*$ and $\rho_0=\frac{1}{2}\sum_{\alpha\in \Delta_0^+}\alpha$.
Let $W_0$ be the subgroup of the Weyl group generated by the reflections
$s_\alpha$ with $\alpha\in\Delta_0$.  It leaves stable $\sigma$.

Consider $\varphi$ a test function of the form
$\varphi_0(\xi_0) \varphi_1(\xi_1)$ with $\varphi_0(\xi_0)=\chi_0(\xi_0)\prod_{\alpha\in \Delta_0^+}\ll\xi_0, H_\alpha\rr$ with $\chi_0(\xi_0)$ a function on $\t_0^*$ with small support near $0$ and identically $1$ near $0$,
while $\varphi_1(\xi_1)$ is supported on a compact subset contained in $\sigma^0$.

For $k$ large,
$$\ll\Theta_a^t,\varphi\rr=\frac{1}{|W|}m_{\sigma}(\lambda,k) \sum_{w\in W_0}
\phi(w(\lambda+\rho)/k).$$
So
$$\ll\Theta_a^t,\varphi\rr=c_0\frac{1}{k^{N_0}}
\sum_{\lambda\in \sigma} m_\sigma(\lambda,k)
\varphi_1((\lambda+\rho_1)/k).$$
As in the preceding case, this implies that $m_\sigma(\lambda,k)=0$ for
$\lambda\in \sigma^0$.
Doing it successively for all $\sigma$ entering in the first sum,
we conclude that $m\in \CS_{\geq 0,\ell-1}(\Lambda)$.
By induction, we conclude that $m=0$.
\end{proof}

\end{document}